\documentclass{amsart}

\usepackage[hidelinks]{hyperref}
\usepackage{mathrsfs}
\usepackage{enumerate}
\usepackage{graphicx,color}
\usepackage{amssymb}
\usepackage{bm}
\usepackage{adjustbox}	
\usepackage{scalerel}	

\newtheorem{defi}{Definition}[section]
\newtheorem{prop}[defi]{Proposition}
\newtheorem{thm}[defi]{Theorem}
\newtheorem{lem}[defi]{Lemma}

\newtheorem{ex}[defi]{Example}

\numberwithin{equation}{section}

\newcommand{\N}{\mathbf{N}}
\newcommand{\Z}{\mathbf{Z}}
\newcommand{\R}{\mathbf{R}}

\newcommand{\cF}{\mathscr F}
\newcommand{\cH}{\mathscr H}		

\newcommand{\cL}{\mathscr L}		

\newcommand{\cR}{\mathscr R}

\newcommand{\cS}{\mathscr S}

\newcommand{\bI}{\operatorname{\mathbf I}}
\newcommand{\bM}{\operatorname{\mathbf M}}
\newcommand{\bN}{\operatorname{\mathbf N}}

\newcommand{\BV}{\operatorname{BV}}

\newcommand{\Lip}{\operatorname{Lip}}

\newcommand{\spt}{\operatorname{spt}}

\newcommand{\dist}{\operatorname{dist}}

\newcommand{\defl}{\mathrel{\mathop:}=}

\newcommand{\im}{\operatorname{im}}

\newcommand{\B}{\textbf B}		
\newcommand{\oB}{\textbf U}		

\newcommand{\curr}[1]{[\![{#1}]\!]}
\newcommand{\id}{{\rm id}}

\newcommand{\res}{\scaleobj{1.7}{\llcorner}}

\makeatletter
\newcommand*{\cone}{%
	{%
		\mathpalette\@coneOf{\times}%
	}%
}
\newcommand*{\@coneOf}[2]{%
	\sbox0{$\m@th#1\mathsf{#2}$}%
	\mathsf{#2}%
	\kern-\wd0 %
	\mkern2.00mu\relax
	\nonscript\mkern0.25mu\relax
	\mathsf{#2}%
}
\makeatother

\title[Lipschitz-volume rigidity of Lipschitz manifolds]{Lipschitz-volume rigidity of Lipschitz manifolds among integral currents}

\author{Roger Z\"{u}st}
\email{
	roger\_zuest@hotmail.com}

\begin{document}
		
\begin{abstract}
We give sufficient conditions such that a volume preserving 1-Lipschitz map from a metric integral current onto an infinitesimally Euclidean Lipschitz manifold is an isometry.
\end{abstract}

\maketitle

\section{Introduction}

The main question investigated in this note is the following: Suppose $(X,d_X,\mu_X)$ and $(Y,d_Y,\mu_Y)$ are metric measure spaces and  $f : X \to Y$ is onto, $1$-Lipschitz and measure preserving. Under what additional assumptions does it follow that $f$ is an isometry?

An elementary example in this direction without any measures is the following rigidity statement: If $X$ is a compact metric space and $f : X \to X$ is onto and $1$-Lipschitz, then $f$ is an isometry, see e.g.\ \cite[Theorem~1.6.15]{BBI}.

If domain and target are not assumed to be equal, then further assumptions are needed for a positive answer. In our setting, the measure preserving property is the main one. It excludes examples like $f$ being a uniform scaling of some domain in $\R^n$ with a scaling factor smaller than one. Among others, a more subtle requirement is that the target $Y$ can't have separated components (at least if metrics are not allowed to take the value $\infty$). 

An instance of Lipschitz-volume rigidity is between Riemannian manifolds and their standard volumes as stated in the following known result, see e.g.\ \cite[Lemma~9.1]{BB}:

\begin{prop}
	\label{prop:riemann}
	Suppose $X$ and $M$ are oriented, closed, connected $C^1$-manifolds equipped with continuous Riemannian metrics. If $f : X \to M$ is $1$-Lipschitz (with respect to the induced length distances), onto and $\operatorname{Vol}(X) \leq \operatorname{Vol}(M)$, then $f$ is an isometry.
\end{prop}

Our main result is a generalization allowing the domain to be a metric integral current in the sense of Ambrosio-Kirchheim \cite{AK}. In this setting the specific definition of volume on the domain becomes crucial and only certain choices lead to isometry. We call these volumes Euclidean rigid. The precise definition is given in \ref{def:rigid}.

\begin{thm}
	\label{thm:rigidity}
	Suppose $m\geq 1$ and $(M,d)$ is a compact, oriented, $m$-dimensional Lipschitz manifold (possibly with boundary $\partial M$) such that:
	\begin{enumerate}[$\qquad$(1)]
		\item $(M\setminus \partial M,d)$ is an essential length space and $(M,d)$ its completion,
		\item $(M,d)$ is infinitesimally Euclidean.
	\end{enumerate}
	Suppose that $T \in \bI_m(X)$ is an $m$-dimensional integral current, $\mu$ is a Finsler volume and $f : \spt(T) \to M$ is a $1$-Lipschitz map such that
	\begin{enumerate}[$\qquad$(a)]
		\item $f_\# T = \curr{M}$, 
		\item $f(\spt(\partial T)) \subset \partial M$,
		\item $\mu$ is Euclidean rigid,
		\item $\bM^\mu(T) \leq \operatorname{Vol}(M)$.
	\end{enumerate}
	Then $f : \spt(T) \to M$ is an isometry and $T = (f^{-1})_\# \curr{M}$.
\end{thm}

This theorem generalizes \cite[Theorem~1.2]{BCS}, \cite[Corollary~1.3]{BCS} and \cite[Theorem~1.1]{DNP} and answers \cite[Question~8.1]{BCS}.

The assumptions can be further weakened. The compactness of $(M,d)$ can be replaced by complete and finite volume and finite boundary volume, so that $\curr M$ is well defined as an integral current.

Here are some clarifications of the terminology used in the statement. Essential length spaces are generalizations of classical length spaces, see Definition~\ref{def:esslength}. In particular, the induced distance of a Riemannian manifold is of this type, see Lemma~\ref{lem:esslength}. Similar to property (ET) in \cite{LW}, we call a countably $\cH^m$-rectifiable metric space $S$ infinitesimally Euclidean if whenever $\varphi : K \to S$ is a bi-Lipschitz chart defined on a compact set $K \subset \R^m$, then the metric derivative $\operatorname{md}(\varphi_x)$ in the sense of \cite{K} is induced by a inner product for almost all $x \in K$. An $m$-dimensional Finsler volume $\mu$ assigns to every norm on $\R^m$ a particular multiple of the Lebesgue measure, see Definition~\ref{def:finslervolume}. $\mu$ then also induces a volume (and a mass $\bM^\mu$) on rectifiable spaces (and rectifiable currents),see Subsection~\ref{sec:finslermass}. For example, the usual mass of rectifiable currents is induced by the Gromov-mass-star Finsler volume. On Riemannian manifolds, or more generally infinitesimally Euclidean spaces, there is only one Finsler volume. We also state an area formula for Finsler volumes between rectifiable spaces, Theorem~\ref{thm:areaformula}. This builds on the known area formula for the Hausdorff measure \cite[Theorem~7]{K}. 


We shortly explain the general strategy of the proof of Theorem~\ref{thm:rigidity}. (a) and (b) can be interpreted as saying that $f$ is a cover of $M$ with algebraic multiplicity $1$. Because $f$ is $1$-Lipschitz, (d) is equivalent to $\bM^\mu(T) = \operatorname{Vol}(M)$, which implies that $f$ is measure preserving in the sense that $\|T\|^\mu(B) = \operatorname{Vol}(f(B))$ for all Borel sets $B \subset X$. This quite readily implies that $f$ is almost injective on $T$ in a measure theoretic sense. The technical part is contained in Proposition~\ref{prop:almostisometry} which guarantees that an almost injective map to $\R^m$ is a locally bi-Lipschitz embedding in case the volume is not distorted too much. The latter is quantified by uniform bounds on Hardy-Littlewood maximal functions of the push-forward measure. The tools used in the proof of Proposition~\ref{prop:almostisometry} are zero-dimensional slices and the connection between normal currents in $\R^m$ and $\BV$-functions. It is a rather direct extraction of the partial rectifiability theorem in the theory of metric currents as stated in \cite[Theorem~7.4]{AK} and \cite[Theorem~7.6]{L}. Working in charts of $M$, Proposition~\ref{prop:almostisometry} implies that $f : \spt(T)\setminus\spt(\partial T) \to M \setminus \partial M$ is a homeomorphism which is locally bi-Lipschitz. Assumption (2) and (c) then further imply that $f$ is an infinitesimal isometry and as a consequence it preserves the length of almost every curve. (1) then allows for a local to global argument to conclude that $f$ is an isometry.

Assumption (1) can't be replaced by the weaker assumption that $(M,d)$ is a length space as shown in Example~\ref{ex:esslength}. Assumption (c) is necessary as seen by linear maps between domains of normed spaces. We show that in particular the Busemann-Hausdorff and the Gromov-mass-star volume are Euclidean rigid, see Lemma~\ref{lem:infirigid} and Lemma~\ref{lem:infirigid2}.

\section{Preliminaries}

\subsection{Metric currents}
Let $X$ be a complete metric space. $\B(x,r)$ denotes the closed ball and $\oB(x,r)$ the open ball around a point $x \in X$ with radius $r > 0$. 

Following the theory of Ambrosio and Kirchheim \cite{AK}, for an integer $m \geq 0$, an $m$-dimensional metric current $T \in \bM_m(X)$ of finite mass in $X$ is a multilinear function $T : \Lip_{\rm b}(X) \times \Lip(X)^m \to \R$ with an associated finite Borel measure $\|T\|$ on $X$. Currents are best understood as a generalization of oriented, compact Riemannian manifolds. For more details and the terminology we refer to \cite{AK}.

According to \cite[Theorem~3.4]{AK}, normal currents $\bN_{m}(\R^m)$ can be identified with the space $\BV(\R^m)$ of functions with bounded variation, i.e.\ those $u \in \text{L}^1(\R^m)$ with
\[
|Du|(U) \defl \sup \left\{\int_U u \operatorname{div}(\varphi) \, d \cL^m \, : \, \varphi \in C^1_c(U,\R^m), \, \|\varphi\|_\infty \leq 1 \right\} < \infty
\]
for all open sets $U \subset \R^m$. Moreover, $\|\curr u\| = \cL^m\res u$ and $\|\partial \curr u\| = |Du|$.

If $\mu$ is a finite Borel measure on $\R^m$, then $M_{\mu} : \R^m \rightarrow [0,\infty]$ denotes the Hardy-Littlewood maximal function defined by
\[
M_\mu(x) \defl \sup_{r > 0} \frac{\mu(\B(x,r))}{\alpha_m r^m} \ ,
\]
where $\alpha_m$ is the (Lebesgue) volume of the Euclidean unit ball in $\R^m$. A covering argument shows that $M_\mu(x) < \infty$ for $\cL^m$-almost every $x \in \R^m$. If $u \in \BV(\R^m)$ and $x,x'\in \R^m$ are Lebesgue points of $u$, then
\begin{equation}
	\label{eq:bvlipschitz}
	|u(x) - u(x')| \leq  c_m \left(M_{|Du|}(x) + M_{|Du|}(x')\right) |x - x'|
\end{equation}
for some constant $c_m > 0$ depending only on $m$. This is a classical result. For proofs, see for example \cite[Lemma~7.1]{L} or \cite[Lemma~7.3]{AK}.

The next result is contained in the statement of \cite[Theorem~7.5]{L} within the theory of local metric currents and follows directly from \cite[Equation~(5.7)]{AK}.

\begin{lem}
	\label{lem:formula}
Suppose $m \geq 1$, $T \in \bN_{m}(X)$, $\pi \in \Lip(X,\R^m)$ and $f \in \Lip(X)$. Then $\pi_\#(T \res f) = \curr{u_f}$ for some $u_f \in \BV(\R^m)$ and
\[
	\langle T, \pi, y \rangle (f) = u_f(y)
\]
for almost every $y \in \R^m$.
\end{lem}

\begin{proof}
Let $\psi \in C_c(\R^m)$ be arbitrary. By \cite[Equation~(5.7)]{AK},
\begin{align*}
\int_{\R^m}\psi(y)\langle T, \pi, y \rangle (f)\,d\cL^m(y) & = T((\psi\circ \pi)\cdot f, \pi) \\
 & = (\pi_\# (T \res f))(\psi, \id_{\R^m}) \\
 & = \int_{\R^m}\psi(y)u_f(y)\,d\cL^m(y) \ .
\end{align*}
Thus $\langle T, \pi, y \rangle (f) = u_f(y)$ for almost every $y \in \R^m$.
\end{proof}

According to \cite[Lemma~4.1]{AK} a subset $S$ of $X$ is countably $\cH^m$-rectifiable if there exist countably many bi-Lipschitz maps $\varphi_i : K_i \to S$ defined on compact subsets $K_i \subset \R^m$ such that the images $\varphi_i(K_i)$ are pairwise disjoint and
\[
\cH^m\biggl(S \setminus \bigcup_i \varphi_i(K_i)\biggr) = 0 \ .
\]
We call such a collection of charts $(\varphi_i,K_i)$ an atlas for $S$.

By \cite[Theorem~4.5]{AK} a current $T \in \bM_m(X)$ is rectifiable if there exists a countably $\cH^m$-rectifiable $S \subset X$ with an atlas $(\varphi_i,K_i)$ and $\theta_i \in \text L^1(K_i)$ for each $i$ such that
\[
\bM(T) = \sum_i \bM(\varphi_{i\#}\curr{\theta_i}) \quad \mbox{and} \quad T = \sum_i \varphi_{i\#}\curr{\theta_i} \ .
\]
The collection of such rectifiable currents is denoted by $\cR_m(X)$. If the densities $\theta_i$ above are in $L^1(K_i,\Z)$, then $T$ is integer rectifiable and their collection is denoted by $I_m(X)$. Moreover, $\bI_m(X) \defl \mathcal I_m(X) \cap \bN_m(X)$ is the collection of integral currents.

There is a canonical choice for the set $S$, namely
\begin{equation}
	\label{def:st}
S_T \defl \{x \in X : \Theta_{*m}(\|T\|,x) > 0\} \ .
\end{equation}
See \cite[Theorem~4.6]{AK}.

\subsection{Finsler Volumes}

Finsler volumes are consistent choices of Haar-measures in normed spaces.

\begin{defi}
	\label{def:finslervolume}
Given an integer $m \geq 1$, an \textbf{$\mathbf m$-dimensional Finsler volume} $\mu$ assigns to every $m$-dimensional normed space $V$ a Haar measure $\mu_V$ with the properties:
\begin{enumerate}
	\item If $A : V\to W$ is linear and short (i.e.\ $\|A\|\leq 1$), then $A$ is volume decreasing, i.e.\ $\mu_W(A(B)) \leq \mu_V(B)$ for all Borel sets $B \subset V$.
	\item If $V$ is Euclidean (i.e.\ the norm is induced by an inner product), then $\mu_V$ is the standard Euclidean volume (the Lebesgue measure with respect to some orthonormal coordinate system).
\end{enumerate}
\end{defi}

Our two primary examples are the Busemann-Hausdorff volume $\mu^{\rm{bh}}$ and the Gromov-mass-star $\mu^{\rm m\ast}$. $\mu^{\rm{bh}}$ agrees with the $m$-dimensional Hausdorff measure and has the defining property that $\mu^{\rm{bh}}_V(\B_V(0,1)) = \alpha_m$, see e.g.\ \cite[Lemma~6]{K}. 
$\mu^{\rm m\ast}$ is defined by $\mu_V^{\rm m \ast}(P) = 2^m$ if $P$ is a parallelepiped of minimal volume that contains $\B_V(0,1)$. 

A special subclass of Finsler volumes is extracted in the following definition. We denote by $|\cdot|$ the standard Euclidean norm on $\R^m$.

\begin{defi}
	\label{def:rigid}
An $m$-dimensional Finsler volume $\mu$ is \textbf{Euclidean rigid} if the following holds: If $\|\cdot\|$ is a norm on $\R^m$ such that 
\begin{enumerate}
	\item $\|\cdot\| \geq |\cdot|$ and
	\item $\mu_{\|\cdot\|} \leq \mu_{|\cdot|}$,
\end{enumerate}
then $\|\cdot\| = |\cdot|$.
\end{defi}

Note that (1) is equivalent to $\id : (\R^m,\|\cdot\|) \to (\R^m,|\cdot|)$ being $1$-Lipschitz. This implies $\mu_{|\cdot|} \leq \mu_{\|\cdot\|}$ by the definition of Finsler volumes. Thus (2) is equivalent to $\mu_{\|\cdot\|} = \mu_{|\cdot|}$.

Due to the properties of Finsler volumes, this definition has the following seemingly more general but equivalent formulation: An $m$-dimensional Finsler volume $\mu$ is Euclidean rigid if and only if the following property holds: If $A : V \to H$ is a linear map from a normed space $V$ into an Euclidean space $H$ of the same dimension $m$ such that $f$ is $1$-Lipschitz and volume preserving (i.e.\ $\mu_V(B) = \mu_H(A(B))$ for all Borel sets $B$), then $A$ is an isometry (i.e.\ $\|v\|_V = \|A(v)\|_H$ for all $v$).

Many definitions of volume have this property. For example the Busemann-Hausdorff and the Gromov-mass-star volume as shown in Lemma~\ref{lem:infirigid} and Lemma~\ref{lem:infirigid2} below. As a consequence also the largest Finsler volume, namely the inscribed Riemannian volume $\mu^{\rm{ir}}$, is Euclidean rigid. See e.g.\ \cite{I} for the precise definition and properties of this volume. $\mu^{\rm{ir}}$ is complemented by the smallest Finsler volume, which we call the circumscribed Riemannian volume $\mu^{\rm{cr}}$. By definition, $\mu^{\rm{cr}}(E) = \alpha_m$ for the (unique) minimal volume ellipsoid $E$ that contains the unit ball of the given normed space. In contrast to the volumes mentioned above, $\mu^{\rm{cr}}$ is not Euclidean rigid as shown in the following example.

\begin{ex}
	\label{ex:outerriem}
Let $|\cdot|$ be the standard Euclidean norm on $\R^2$ with unit disk $B$ and let $C \subset B$ be a regular $2n$-con ($n \geq 2$) with vertices on the unit circle $\partial B$. By symmetry, $B$ is the ellipse of minimal area that contains $C$. Let $\|\cdot\|$ be the norm of $\R^2$ for which $C$ is the unit disk. Then $\id : (\R^2,\|\cdot\|) \to (\R^2,|\cdot|)$ is $1$-Lipschitz and volume preserving $\mu^{\rm{cr}}_{\|\cdot\|} = \mu^{\rm{cr}}_{|\cdot|}$, but the two norms are obviously not equal.
\end{ex}

\begin{lem}
	\label{lem:infirigid}
The Busemann-Hausdorff volume $\mu^{\rm{bh}}$ is Euclidean rigid.
\end{lem}

\begin{proof}
Let $\|\cdot\|$ be a norm on $\R^m$ with properties (1) and (2) of Definition~\ref{def:rigid}. The Busemann-Hausdorff volume has the defining property
\[
	\mu^{\rm{bh}}_{\|\cdot\|}(\B_{\|\cdot\|}(0,1)) = \alpha_m =  \mu^{\rm{bh}}_{|\cdot|}(\B_{|\cdot|}(0,1)) \ .
\]
Since $\B_{\|\cdot\|}(0,1) \subset \B_{|\cdot|}(0,1)$ by (1) and $\mu^{\rm{bh}}_{\|\cdot\|} \leq \mu^{\rm{bh}}_{|\cdot|}$ by (2), it holds
\begin{align*}
	\mu^{\rm{bh}}_{\|\cdot\|}(\B_{\|\cdot\|}(0,1)) \leq \mu^{\rm{bh}}_{|\cdot|}(\B_{\|\cdot\|}(0,1)) \leq \mu^{\rm{bh}}_{|\cdot|}(\B_{|\cdot|}(0,1)) = \mu^{\rm{bh}}_{\|\cdot\|}(\B_{\|\cdot\|}(0,1)) \ .
\end{align*}
Thus equality holds and hence $\B_{\|\cdot\|}(0,1) = \B_{|\cdot|}(0,1)$, or equivalently, $\|\cdot\| = |\cdot|$.
\end{proof}

\begin{lem}
	\label{lem:infirigid2}
The Gromov-mass-star volume $\mu^{\rm{m}\ast}$ is Euclidean rigid.
\end{lem}

\begin{proof}
Let $\|\cdot\|$ be a norm on $\R^m$ with properties (1) and (2) of Definition~\ref{def:rigid}. Let $e_1,\dots,e_m$ be any orthonormal basis with respect to the standard Euclidean norm $|\cdot|$ and denote by $f_1,\dots,f_m$ the dual basis. The unit ball $\B_{|\cdot|}(0,1)$ is contained in the parallelepiped
\[
	P \defl \{x \in \R^m : |f_i(x)| \leq 1, \mbox{ for all } i\}
\]
of (Lebesgue) volume $2^m$. It holds $\B_{\|\cdot\|}(0,1) \subset P$ because $\|\cdot\| \geq |\cdot|$ and $\mu^{\rm{m}\ast}_{\|\cdot\|}(P) = 2^m$ because $\mu_{|\cdot|}^{\rm{m}\ast} = \mu_{\|\cdot\|}^{\rm{m}\ast}$ by assumption. By the definition of the Gromov-mass-star volume, $P$ must be a parallelepiped of minimal volume that contains $\B_{\|\cdot\|}(0,1)$. The only points $x \in \B_{|\cdot|}(0,1)$ with $|f_1(x)| = 1$ are $e_1$ and $-e_1$. Because $\B_{\|\cdot\|}(0,1) \subset \B_{|\cdot|}(0,1)$, the points $\pm e_1$ have to be contained in $\B_{\|\cdot\|}(0,1)$ too. Otherwise, $P$ can be scaled in direction $e_1$ to a parallelepiped of smaller volume that also contains $\B_{\|\cdot\|}(0,1)$. This is a contradiction. Since the orthonormal basis $e_1,\dots,e_m$ is arbitrary, we conclude that every point in the sphere $\partial \B_{|\cdot|}(0,1)$ is contained in $\B_{\|\cdot\|}(0,1)$. Hence $\B_{\|\cdot\|}(0,1) = \B_{|\cdot|}(0,1)$, respectively, $\|\cdot\| = |\cdot|$.
\end{proof}

\subsection{Finsler Mass}
\label{sec:finslermass}

As in \cite[Definition~2.4]{Z} any Finsler volume induces a notion of volume on rectifiable spaces and mass on rectifiable metric currents. To recall this definition, we first need the Jacobian of seminorms.

If $s$ is a seminorm on $\R^m$, the Jacobian $\mathbf{J}^{\mu}(s)$ of $s$ with respect to $\mu$ is $\mu_s([0,1]^m)$ if $s$ is a norm, and $0$ otherwise. Equivalently, in case $s$ is a norm,
\begin{equation}
	\label{eq:defjacobian}
	\mathbf{J}^{\mu}(s) = \frac{\mu_s(B)}{\cL^m(B)}
\end{equation}
for a (any) Borel set $B \subset \R^m$ of positive and finite measure.

Suppose $S \subset X$ is a countably $\cH^m$-rectifiable with atlas $(\varphi_i,K_i)$. We may isometrically embed $X$ into $\ell_\infty(X)$ and extend each $\varphi_i$ to a Lipschitz map defined on all of $\R^m$. Due to \cite{K}, the metric derivative of $\varphi_i$ is defined at almost every point $x \in K_i$ and denoted by $\operatorname{md}(\varphi_i)_x$. A metric derivative is a seminorm on $\R^m$. We note that a different isometric embedding into $\ell_\infty$ and a different Lipschitz extension of $\varphi_i$ changes $\operatorname{md}(\varphi_i)$ in at most a set of measure zero. These choices could also be bypassed by switching to approximate limits in the definition of metric derivatives similar to \cite[\S3.1.2]{F} in the classical case.

Suppose $\mu$ is an $m$-dimensional Finsler volume, then $\mu_S$ is the Borel measure on $S$ defined by
\begin{equation}
	\label{eq:musdef}
\mu_S(B) \defl \sum_i \int_{K_i \cap \varphi_i^{-1}(B)} \mathbf J^\mu(\operatorname{md}((\varphi_i)_x))\,d\cL^m(x) \ .
\end{equation}

Suppose $T \in \cR_m(X)$ is induced by an atlas with densities $(\varphi_i,K_i,\theta_i)$, then the Borel measure and mass of $T$ with respect to $\mu$ is defined by
\[
\|T\|^\mu(B) \defl \sum_i \int_{K_i \cap \varphi_i^{-1}(B)} |\theta_i(x)| \mathbf J^\mu(\operatorname{md}(\varphi_i)_x)\,d\cL^m(x) \ ,
\]
and $\bM^\mu(T) \defl \|T\|^\mu(X)$. Similarly we could define a mass on rectifiable $G$-chains as introduced in \cite{DPH}.

Essentially by \cite[Lemma~2.5]{Z} it holds:
\begin{itemize}
	\item This extended notion of volume and mass on rectifiable spaces is compatible with the underlying Finsler volume on normed spaces.
	\item The Gromov-mass-star measure $\|T\|^{\rm m\ast}$ is the usual measure $\|T\|$ for metric currents, whereas the mass on rectifiable $G$-chain is induced by the Busemann-Hausdorff volume.
	\item Finsler volumes are comparable in the sense that
	\[
	C_m^{-1}\|T\|^{\mu_2} \leq \|T\|^{\mu_1} \leq C_m\|T\|^{\mu_2}
	\]
	for some universal $C_m \geq 1$. In particular, any such measure is comparable to the Ambrosio-Kirchheim or the Hausdorff mass.
\end{itemize}
The last point indicates in particular that questions concerning measurability, integrability and null-sets are independent of the choice of Finsler volume.

Suppose $f : X \to Y$ is Lipschitz and $S \subset X$ is countably $\cH^m$-rectifiable with atlas $(\varphi_i,K_i)$ as above. For $\cH^m$-almost every $x \in S$ there exists $i$ and $y \in K_i$ with $\varphi_i(y) = x$ such that the $\mu$-Jacobian of $f$ at $x$ is well-defined by
\[
\mathbf J^\mu(\operatorname{md}(f_{x})) \defl \frac{\mathbf J^\mu(\operatorname{md}((f \circ \varphi_i)_y))}{\mathbf J^\mu(\operatorname{md}((\varphi_i)_y))} \ .
\]
Note that $\operatorname{md}((\varphi_i)_y))$ is indeed a norm (not a degenerated seminorm) for almost every $x \in K_i$ by the area formula with respect to the Hausdorff measure $\cH^m$, \cite[Theorem~7]{K}, and the fact that bi-Lipschitz maps preserve $\cH^m$-null sets. This definition is independent of the underlying atlas in the sense that for another atlas the two definitions agree in the complement of an $\cH^m$-null subset of $S$. This boils down to 
the following chain rule for Jacobians. Assume $(\varphi,K_1)$ and $(\psi,K_2)$ are two charts of $S$ with the same image. Then for almost every $x \in K_1$,
\begin{align*}
\mathbf J^\mu(\operatorname{md}((f \circ \varphi)_x)) & = \mathbf J^\mu\bigl(\operatorname{md}((f \circ \psi \circ \psi^{-1} \circ\varphi)_x)\bigr) \\
 & = \mathbf J^\mu\bigl(\operatorname{md}((f \circ \psi)_{(\psi^{-1} \circ\varphi)(x)}) \circ D(\psi^{-1} \circ\varphi)_x \bigr) \\
 & = \mathbf J^\mu\bigl(\operatorname{md}((f \circ \psi)_{(\psi^{-1} \circ\varphi)(x)})\bigr) \, \bigl|\det(D(\psi^{-1} \circ\varphi)_x)\bigr| \ .
\end{align*}
If we plug in $f = \id$ we also obtain
\begin{align*}
\frac{\mathbf J^\mu(\operatorname{md}(\varphi_x))} {\mathbf J^\mu\bigl(\operatorname{md}(\psi_{(\psi^{-1} \circ\varphi)(x)})\bigr)} & = \bigl|\det(D(\psi^{-1} \circ\varphi)_x)\bigr| \ ,
\end{align*}
and hence
\[
\frac{\mathbf J^\mu(\operatorname{md}((f \circ \varphi)_x))}{\mathbf J^\mu(\operatorname{md}(\varphi_x))} = \frac{\mathbf J^\mu\bigl(\operatorname{md}((f \circ \psi)_{(\psi^{-1} \circ\varphi)(x)})\bigr)} {\mathbf J^\mu\bigl(\operatorname{md}(\psi_{(\psi^{-1} \circ\varphi)(x)})\bigr)}
\]
for almost all $x \in K_1$.

\begin{thm}[Finsler area formula]
	\label{thm:areaformula}
Suppose that $X$ and $Y$ are complete metric spaces, $S \subset X$ is countably $\cH^m$-rectifiable, $f : X \to Y$ is Lipschitz and $g : S \to [0,\infty]$ is measurable. Then
\[
\int_S g(x) \mathbf J^\mu(\operatorname{md}(f_x)) \,d\mu_S(x) = \int_{f(S)} \Biggl(\sum_{x \in f^{-1}(y)} g(x) \Biggr) \,d\mu_{f(S)}(y) \ .
\]
\end{thm}

\begin{proof}
Let $(\varphi_i,K_i)$ be an atlas of $S$. By the definition of $\mu_S$ in \eqref{eq:musdef} and approximating $x \mapsto g(x) \mathbf J^\mu(\operatorname{md}(f_x))$ by simple functions, the left-hand side above can be written as
\begin{align*}
 & \sum_i \int_{K_i} g(\varphi_i(z)) \mathbf J^\mu(\operatorname{md}(f_{\varphi_i(z)})) \mathbf J^\mu(\operatorname{md}((\varphi_i)_z)) \,d\cL^m(z) \\
 & \qquad  = \sum_i \int_{K_i} g(\varphi_i(z)) \mathbf J^\mu(\operatorname{md}((f\circ\varphi_i)_{z})) \,d\cL^m(z) \ .
\end{align*}
To conclude the proof we assume first that $g = \chi_A$ for some measurable subset $A \subset S$. Proceeding as in Theorem~7 in \cite{K} we find for each $i$ countably many disjoint compact subsets $E_{i,j} \subset K_i$ such that:
\begin{itemize}
	\item $f \circ \varphi_i : E_{i,j} \to f(S)$ is bi-Lipschitz and $\operatorname{md}(f\circ \varphi_i)$ is a norm almost everywhere.
	\item $\operatorname{md}(f\circ \varphi_i)$ is a strict seminorm for almost all points of $K_i \setminus \bigcup_j E_{i,j}$.
	\item $\cH^m((f\circ\varphi_i)(K_i \setminus \bigcup_j E_{i,j})) = 0 = \mu_{f(S)}((f\circ\varphi_i)(K_i \setminus \bigcup_j E_{i,j}))$.
\end{itemize}
By the first point, $(f \circ \varphi_i, E_{i,j})$ is a chart of $f(S)$ for each $i$ and $j$ and by the definition of $\mu_{f(S)}$ in \eqref{eq:musdef}, 
\begin{align*}
\int_{E_{i,j} \cap \varphi_i^{-1}(A)} \mathbf J^\mu(\operatorname{md}((f\circ \varphi_i)_z))\,& d\cL^m(z) = \mu_{f(S)}(f(A \cap \varphi_i(E_{i,j}))) \\
 & = \int_{f(S)} \Biggl(\sum_{x \in f^{-1}(y)} \chi_{A \cap  \varphi_i(E_{i,j})}(x) \Biggr) \,d\mu_{f(S)}(y) \ .
\end{align*}
Note that the integrand in second line is equal to the characteristic function of $f(A \cap \varphi_i(E_{i,j}))$. Let $E \defl \bigcup_{i,j} \varphi_i(E_{i,j}) \subset S$. From the third point it follows $\mu_{f(S)}(f(S \setminus E))=0$. By summing over all $i$ and $j$ using the second point in the second line below:
\begin{align*}
\int_A \mathbf J^\mu(\operatorname{md}(f_x)) \,d\mu_S(x) & = \sum_i \int_{K_i \cap \varphi_i^{-1}(A)} \mathbf J^\mu(\operatorname{md}((f\circ\varphi_i)_{z})) \,d\cL^m(z) \\
 & = \sum_{i,j} \int_{E_{i,j} \cap \varphi_i^{-1}(A)} \mathbf J^\mu(\operatorname{md}((f\circ\varphi_i)_{z})) \,d\cL^m(z) \\
 & = \sum_{i,j} \int_{f(S)} \Biggl(\sum_{x \in f^{-1}(y)} \chi_{A  \cap \varphi_i(E_{i,j})}(x) \Biggr) \,d\mu_{f(S)}(y) \\
 & = \int_{f(S)} \Biggl(\sum_{x \in f^{-1}(y)} \chi_{A \cap E}(x) \Biggr) \,d\mu_{f(S)}(y) \\
 & = \int_{f(S) \setminus f(S\setminus E)} \Biggl(\sum_{x \in f^{-1}(y)} \chi_A(x) \Biggr) \,d\mu_{f(S)}(y) \\
 & = \int_{f(S)} \Biggl(\sum_{x \in f^{-1}(y)} \chi_A(x) \Biggr) \,d\mu_{f(S)}(y) \ .
\end{align*}
So the theorem holds for simple functions $g$ and by approximation for all non-negative measurable functions.
\end{proof}

A map $f : X \to Y$ as in the theorem above is called an infinitesimal isometry on $S$ if whenever $(\varphi,K)$ is a chart of $S$, then $\operatorname{md}(\varphi_x) = \operatorname{md}(f \circ \varphi_x)$ for almost all $x \in K$.

\subsection{Essential length spaces}

The essential length distance originates in \cite{DCP}. Our formulation in the context of metric measure spaces is adapted from \cite{AHPS}.

\begin{defi}
	\label{def:esslength}
A metric measure space $(X,d,\mu)$ is an \textbf{essential length space} if for all $x,y \in X$, all $N \subset X$ with $\mu(N) = 0$ and all $\epsilon > 0$ there exists a Lipschitz curve $\gamma : [0,1] \to X$ connecting $x$ and $y$ such that $\cL^1(\gamma^{-1}(N)) = 0$ and
\[
	d(x,y) + \epsilon \geq L(\gamma) \ .
\]
In other words, $d(x,y)$ is equal to the \textbf{essential length distance}
\[
d_{\rm{ess}}(x,y) \defl \sup \{d_N(x,y) \, : \, N \subset X, \mu(N) = 0\} \ ,
\]
where
\[
d_N(x,y) \defl \inf \left\{ L(\gamma)\,:\, \gamma \in \Lip([0,1], X), \gamma(0) = x, \gamma(1) = y, \cL^{1}(\gamma^{-1}(N)) = 0 \right\} \ .
\]
\end{defi}

This is compatible with the definition of essential metric in \cite[Definition~4.1]{AHPS} due to \cite[Proposition~4.6]{AHPS}. A further generalization to $p$-essential length distances for $p < \infty$ is studied in \cite{CS}.

Essential length spaces are obviously standard length spaces but the converse does not hold even for quite nice Lipschitz manifolds as we will see in Example~\ref{ex:esslength}. An $m$-dimensional Lipschitz manifold (possibly with boundary) is a metric space $(M,d)$ which can be covered by open sets which are bi-Lipschitz equivalent to open subsets of $\R^m$ (or of $\mathbf H^m \defl \{x \in \R^m : x_m \geq 0\}$). See e.g.\ \cite{LV} for more details. It is understood that in this case $\mu$ is the Hausdorff measure on $M$ (but any other Finsler volume induces the same essential length distance). If $(M,d)$ is a Lipschitz manifold, then the essential distance $d_{\rm{ess}}$ is locally bi-Lipschitz equivalent to $d$. This follows by elementary means and is an instance of \cite[Theorem~3.1]{DCJS} which lists five other characterizations for this bi-Lipschitz equivalence for more general metric measure spaces.

Suppose $M$ is a $C^1$-manifold (possibly with boundary) and $g$ is a continuous Riemannian metric on $M$. The induced length distance $d_{\rm i}$ is defined by
\[
d_{\rm i}(x,y) \defl \inf_{\gamma} L(\gamma) \ ,
\]
where the infimum is over length of all piecewise $C^1$-curves $\gamma : [0,1] \to M$ with $\gamma(0) = x$ and $\gamma(1) = y$. The length of such a $\gamma$ is defined by
\[
L(\gamma) \defl \int_0^1 g_{\gamma(t)}(\gamma'(t),\gamma'(t))^\frac{1}{2}\,dt \ .
\]
As shown in \cite[Corollary~3.13]{B}, $d_{\rm i}$ can equivalently be defined with respect to absolutely continuous curves instead of piecewise $C^1$-curves. 
If $d$ is a background metric on $M$ such that $C^1$-charts on $M$ are locally bi-Lipschitz, such  metric exists by \cite[Theorem~3.5]{LV}, then absolutely continuous curves in $M$ are those curves absolutely continuous with respect to $d$. Thus any curve class in between piecewise $C^1$ and absolutely continuous induces the same length metric $d_{\rm i}$, see \cite[\S3.6]{B}. In contrast to general Lipschitz manifolds, Riemannian manifolds with the induced length distance are essential length spaces. The main reason is that for any sequence of $C^1$-curves $(\gamma_n)$ which converges in the $C^1$-topology to $\gamma$, it holds $L(\gamma_n) \to L(\gamma)$. Here are the details.

\begin{lem}
	\label{lem:esslength}
If $M$ is a $C^1$-manifold with continuous Riemannian metric $g$, then
\[
d_{\rm{ess}} = d_{\rm i} \ .
\]
\end{lem}

\begin{proof}
We assume that $M$ is connected since for points in different components, $d_{\rm{ess}}$ and $d_{\rm i}$ are $\infty$. Let $\gamma : [0,1] \to M$ be an injective piecewise $C^1$-curve which connects $x = \gamma(0)$ and $y = \gamma(1)$ in $M$. Working in a chart we first assume that $M$ is an open, connected subset $U$ of $\R^m$. Let $N \subset U$ be a set of $\cL^m$-measure zero and fix $\epsilon > 0$. By a smoothing argument we may replace $\gamma$ by a $C^1$-embedding $\tilde \gamma : [0,1] \to M$ connecting $x$ and $y$ such that $L(\tilde \gamma) < L(\gamma) + \epsilon$. In dimension $1$ this is trivially true, in dimensions $\geq 2$ this follows by a general position argument smoothing the corners. By the tubular neighborhood theorem, there exists a $C^1$-embedding $\Gamma : [0,1] \times \oB^{m-1}(0,1) \to M$ with $\Gamma(t,0) = \tilde\gamma(t)$. Define $\delta : [0,1] \times \oB^{m-1}(0,1) \to [0,1]$ by $\delta(t,p) \defl t(1-t)p$ and $\tilde \Gamma(t,p) \defl \Gamma(t,\delta(t,p))$. Then $\tilde \Gamma$ is a $C^1$-map which is still a $C^1$-embedding on $(0,1) \times \oB^{m-1}(0,1)$ but the endpoints $x$ and $y$ on $t=0,1$ are fixed. Applying the Theorem of Fubini there exists a sequence $p_n \to 0$ in $\oB^{m-1}(0,1)$ such that each curve $\gamma_{n,N}(t) \defl \tilde \Gamma(t,p_n)$ satisfies $\cL^1(\gamma_{n,N}^{-1}(N)) = 0$. $\gamma_n$ converges in the $C^1$-norm to $\tilde \gamma$, hence $L(\gamma_{n,N}) \to L(\tilde\gamma)$ for $n \to \infty$. Thus for $n$ big enough $\gamma_{n,N}$ is a $C^1$-curve connecting $x$ with $y$, essentially avoiding $N$ such that $L(\gamma_{n,N}) < L(\tilde\gamma) + \epsilon$. By \cite[Theorem~4.11]{B}, the length $L(\gamma)$ for an absolutely continuous curve $\gamma$ agrees with the metric definition of length with respect to the induced length distance $d_{\rm i}$. We denote this length by $L_{d_{\rm i}}(\gamma)$.

For arbitrary $x$ and $y$ (possibly on the boundary) let $\gamma : [0,1] \to M$ be piecewise $C^1$-curve connecting them such that $L(\gamma) \leq d_{\rm i}(x,y) + \epsilon$ and let $N \subset M$ be a set of measure zero. By approximation we may assume that $\gamma((0,1)) \subset M\setminus \partial M$. Covering $\gamma((0,1))$ by countably many charts in $M\setminus \partial M$ we find a $C^1$-embedding $\gamma_{\epsilon,N} : [0,1] \to M$ with:
\begin{enumerate}
	\item $\gamma_{\epsilon,N}(0) = x$, $\gamma_{\epsilon,N}(1) = y$, $\gamma_{\epsilon,N}((0,1)) \subset M\setminus \partial M$,
	\item $\cL^1(\gamma_{\epsilon,N}^{-1}(N)) = 0$,
	\item $L(\gamma_{\epsilon,N}) \leq L(\gamma) + \epsilon$.
\end{enumerate}
It follows that
\begin{align*}
d_{\rm{ess}}(x,y) & = \sup_{\cH^m(N)=0} L_{d_{\rm i}}(\gamma_{\epsilon,N}) = \sup_{\cH^m(N)=0} L(\gamma_{\epsilon,N}) \\
 & \leq L(\gamma) + \epsilon = L_{d_{\rm i}}(\gamma) + \epsilon \leq d_{\rm i}(x,y) + 2\epsilon \ .
\end{align*}
$d_{\rm{ess}}(x,y) \geq d_{\rm i}(x,y)$ is clear by definition. This proves the statement.
\end{proof}

The result above as well as those in \cite{B} should hold as well for continuous Finsler metrics $g$ on $C^1$-manifolds. In this situation, $g$ assigns to any point $p \in M$ a norm in the tangent space $g_p : T_pM \to [0,\infty)$ such that $g\circ X$ is continuous for every continuous vector field $X$ on $M$.

\section{Proof of the main theorem}

The following proposition is motivated by the partial rectifiability theorems \cite[Theorem~7.6]{L} and \cite[Theorem~7.4]{AK}. Since rectifiable currents are concentrated on separable spaces we may assume, by restricting to the support, that the ambient space is complete and separable.

\begin{prop}
	\label{prop:almostisometry}
Suppose $m \geq 1$, $L > 0$, $T \in \bI_m(X)$, $\pi \in \Lip(X,\R^m)$ and $V \subset \R^m$ is open. Set $U \defl \pi^{-1}(V) \subset X$ and assume that the following assumptions hold:
\begin{enumerate}
	\item $U \cap \spt(\partial T) = \emptyset$,
	\item $\pi_\#(T \res U) = \curr{V}$,
	\item $\pi : U \to V$ is almost injective in the sense that $\pi^{-1}(y) \cap S_T$ consists of a single point for almost all $y \in V$. ($S_T$ is as in \eqref{def:st}.)
	\item $\Lip(\pi)^{m-1} M_{\pi_\#\|T \res U\|}(y) \leq L$ for almost all $y \in V$ (hence for all $y \in V$).
\end{enumerate}
Then $\pi : \spt(T) \cap U \to V$ is a homeomorphism which is locally bi-Lipschitz in the sense that
\[
\Lip(\pi)^{-1} |\pi(x) - \pi(x')| \leq d(x,x') \leq 2 c_m L |\pi(x) - \pi(x')|
\]
for all $x,x' \in \spt(T) \cap U$ with
\begin{equation}
	\label{eq:separated}
d(x,x') < \min(\dist(x, X \setminus U), \dist(x', X \setminus U))
\end{equation}
Here, $c_m > 0$ is the constant of \eqref{eq:bvlipschitz}.
\end{prop}

\begin{proof}

We abbreviate $\mu \defl \Lip(\pi)^{m-1}\pi_\#\|T \res U\|$, i.e.
\[
\mu(B) = \Lip(\pi)^{m-1}\|T\|(\pi^{-1}(B) \cap U)
\]
for every Borel set $B \subset \R^m$. So (4) is equivalent to $M_{\mu}(y) \leq L$.

By \cite[Theorem~5.7]{AK}, (2) and (3) there is a Borel set of full measure $A \subset V$ such that for any $y \in A$ there exists a unique point $x(y) \in \pi^{-1}(y) \cap S_T$ with $\langle T, \pi, y \rangle = \curr{x(y)} \in \bI_0(X)$ and $\pi_\#\langle T, \pi, y \rangle = \curr y$. Note that every element of $\bI_0(X)$ is a finite sum $\sum_i n_i\curr{x_i}$ with integer multiplicities $n_i \in \Z$. 

For $f \in \Lip(X)$ with $\spt(f) \subset U$ we let $u_f \in \BV(\R^m)$ be the function that represents the current $\pi_\#(T \res f)$ as in Lemma~\ref{lem:formula}. For any Borel set $B \subset \R^m$,
\begin{align*}
|Du_f|(B) & = \|\partial \curr{u_f}\|(B) = \|\partial(\pi_{\#}(T \res f))\|(B) = \| \pi_{\#}(\partial(T \res f))\|(B) \\
 & \leq  \Lip(\pi)^{m-1} \|\partial(T \res f)\|(\pi^{-1}(B)) \ .
\end{align*}
If moreover $\Lip(f) \leq 1$, it follows from (1) and \cite[Equation~(3.5)]{AK} that
\[
\partial(T \res f) = (\partial T) \res f - T \res df =  - T \res df = - (T\res U)\res df \ .
\]
Hence
\begin{align}
	\nonumber
|Du_f|(B) & \leq \Lip(\pi)^{m-1}\|T \res \spt(f)\|(\pi^{-1}(B)) \\
	\label{eq:variation}
 & \leq  \Lip(\pi)^{m-1}\|T \res U\|(\pi^{-1}(B)) \leq \mu(B) \ .
\end{align}
Let $\cF \subset \Lip(X)$ be a countable collection of $1$-Lipschitz functions such that for every $x \in U$, $0 < \epsilon < 1$ and $0 < \rho < \dist(x,X\setminus U)$ there is $f_{x,\rho,\epsilon} \in \cF$ with
\begin{equation}
	\label{eq:fproperties}
f_{x,\rho,\epsilon}(x) \geq \epsilon \rho \ , \quad 0 \leq f_{x,\rho,\epsilon} \leq \rho \ , \quad f_{x,\rho,\epsilon} = 0 \text{ on } X \setminus \oB(x,\rho) \ .
\end{equation}
Note that $X$ and hence also $U$ is separable. Then there exists a Borel set $A' \subset A$ of full measure such that for every $f \in \cF$ every $y \in A'$ is a density point of $u_f$ and $\langle T, \pi, y \rangle(f) = u_f(y)$. From \eqref{eq:bvlipschitz}, \eqref{eq:variation} and (4) it follows that
\begin{align}
	\nonumber
|u_f(y) - u_f(y')| & \leq c_m \left(M_\mu(y) + M_\mu(y') \right) |y - y'| \\
 \label{eq:lipbound}
 & \leq 2c_mL|y-y'|
\end{align}
for all $y,y' \in A'$ and $f \in \cF$. Let $x,x' \in \pi^{-1}(A') \cap S_T$ be different but close enough together such that \eqref{eq:separated} holds and set $\rho \defl d(x,x')$. Fix $0 < \epsilon < 1$ and set $f \defl f_{x,\rho,\epsilon} \in \cF$ as in \eqref{eq:fproperties}. It holds $f(x) \geq \epsilon \rho$ and $f(x') = 0$. Then
\begin{align*}
|\langle T, \pi, \pi(x') \rangle(f)| & \leq \int_X f \, d\|\langle T, \pi, \pi(x') \rangle\| \\
 & \leq \rho \cdot \|\langle T, \pi, \pi(x') \rangle\|(\{x\}) = 0 \ .
\end{align*}
On the other side, 
\begin{align*}
|\langle T, \pi, \pi(x) \rangle(f)| & \geq |\langle T, \pi, \pi(x) \rangle(\chi_{\{x\}} f)| - |\langle T, \pi, \pi(x) \rangle(\chi_{U \setminus \{x\}} f)| & \\
 & \geq \epsilon \rho |\langle T, \pi, \pi(x) \rangle(\chi_{\{x\}})| \geq \epsilon \rho \ .
\end{align*}
We conclude from \cite[Equation~(5.7)]{AK} and \eqref{eq:lipbound} that
\begin{align*}
\epsilon d(x,x') & = \epsilon \rho \leq |\langle T, \pi, \pi(x) \rangle(f) - \langle T, \pi, \pi(x') \rangle(f)| \\
 & = |u_f(\pi(x)) - u_f(\pi(x'))| \\
 & \leq 2c_m L |\pi(x) - \pi(x')|.
\end{align*}
This holds for all $0 < \epsilon < 1$, hence
\begin{eqnarray}
	\label{eq:lipbound2}
d(x,x') \leq 2c_m L |\pi(x) - \pi(x')|
\end{eqnarray}
for all $x,x' \in \pi^{-1}(A') \cap S_T$ separated as in \eqref{eq:separated}.

We claim that $T \res U$ is concentrated on $\pi^{-1}(A') \cap S_T$. For any $\epsilon > 0$, $V \setminus A'$ can be covered by countably many balls $\B(y_n,r_n) \subset V$ such that
\[
\sum_n \alpha_m r^m_n \leq \epsilon \ .
\]
By (4),
\begin{align*}
\|T\res U\|(\pi^{-1}(V \setminus A')) & \leq \sum_n\|T\res U\|(\pi^{-1}(\B(y_n,r_n))) \\
 & \leq \sum_n M_{\pi_\#\|T\res U\|}(y_n) \alpha_m r^m_n \\
 & \leq \Lip(\pi)^{1-m}L\epsilon \ .
\end{align*}
Hence $\|T\res U\|(\pi^{-1}(V \setminus A')) = 0$ and thus $\pi^{-1}(A') \cap S_T$ is dense in $\spt(T \res U) = \spt(T) \cap U$. This shows that
\begin{equation}
	\label{eq:bilip}
d(x,x') \leq 2 c_m L |\pi(x) - \pi(x')|
\end{equation}
for all $x,x' \in \spt(T) \cap U$ separated as in \eqref{eq:separated}.

Next we show that $\pi : \spt(T) \cap U \to V$ is an open map. Let $x \in \spt(T) \cap U$ and $r > 0$ be small enough such that:
\begin{itemize}
	\item $\B(x,r) \subset U$,
	\item $T \res \oB(x,r) \in \bI_m(X)$,
	\item $\spt(\partial(T \res \oB(x,r))) \subset \mathbf S(x,r) \defl \{x' \in X : d(x,x') = r\}$,
	\item $\pi$ restricted to $\spt(T) \cap \oB(x,r)$ is a bi-Lipschitz embedding into $\R^m$.
\end{itemize}
If $d_x$ denotes the distance function to $x$, then point two and three are consequences of the slicing identity
\[
\langle T, d_x, r \rangle = \partial (T \res \{d_x < r\}) - (\partial T)\res \{d_x < r\} = \partial (T \res \{d_x < r\})
\]
for almost all small $r > 0$ which is due to \cite[Lemma~5.3]{AK} or \cite[Definition~6.1]{L} and assumption (1). Point four is a consequence of \eqref{eq:bilip}. It follows that $R \defl \pi_\#(T \res \oB(x,r)) \neq 0$ because $x \in \spt(T)$ and $\pi$ is bi-Lipschitz. Now $\oB(\pi(x),s) \subset V \setminus \pi(\mathbf S(x,r))$ for some small $s > 0$. If $R \res \oB(\pi(x),s) \neq 0$, then $R \res \oB(\pi(x),s) = \lambda \curr{\oB(\pi(x),s)}$ for some $\lambda \neq 0$ by the constancy theorem of \cite{F}. Thus $\oB(\pi(x),s)$ is contained in $\pi(\B(x,r))$. Otherwise if $R \res \oB(\pi(x),s) = 0$,
\[
0 = R \res \oB(\pi(x),s) = \pi_\#(T \res N_x)
\]
for the neighbourhood $N_x \defl \spt(T) \cap \oB(x,r) \cap \pi^{-1}(\oB(\pi(x),s))$ of $x$ in $\spt(T)$. But $\pi$ is bi-Lipschitz on $N_x$ and $T \res N_x \neq 0$ because $x$ is in the support of $T$. Thus $\pi_\#(T \res N_x) \neq 0$. This is a contradiction.

It remains to show that $\pi : \spt(T) \cap U \to V$ is injective. But this follows directly from assumption (3) and the openness of $\pi$.
\end{proof}

This allows to prove the main theorem. Without loss of generality we assume that $\spt(T) = X$.

\begin{proof}[{Proof of Theorem~\ref{thm:rigidity}}]
Let $f : X \to M$ as in the statement. $f$ being $1$-Lipschitz and assumption (c) imply that $f$ is volume preserving in the sense
\begin{equation}
	\label{eq:volpreserve}
\|T\|^\mu(B) = \|\curr M\|^\mu(f(B)) = \cH^m(f(B))
\end{equation}
for any Borel set $B \subset X$. $T$ is represented by the countably $\cH^m$-rectifiable set $S = S_T$, a density $\theta : S \to \N$ and an orientation induced by an atlas of positively oriented, pairwise disjoint charts $(\varphi_i,K_i)$ for $S$. With the Finsler area formula \ref{thm:areaformula} it follows
\begin{align*}
\bM^\mu(T) & = \int_S \theta(x) \,d\mu_S(x) \geq \int_S \theta(x) \mathbf J^\mu(\operatorname{md}(f_x)) \,d\mu_S(x) \\
 & = \int_{M} \Biggl(\sum_{x \in f^{-1}(y)} \theta(x) \Biggr) \,d\mu_{M}(y) \geq \bM^\mu(\curr M) = \bM^\mu(T) \ .
\end{align*}
The first inequality holds because $f$ is $1$-Lipschitz and hence $\mathbf J^\mu(\operatorname{md}(f_x)) \leq 1$ almost everywhere. Thus we obtain equalities and the following consequences:
\begin{itemize}
	\item $\cH^0(f^{-1}(y) \cap S) = 1$ for almost all $y \in M$.
	\item $\mathbf J^\mu(\operatorname{md}(f_x)) = 1$ for $\cH^m$-almost all $x \in S$.
	\item $\theta(x) = 1$ for $\cH^m$-almost all $x \in S$.
\end{itemize}
Because $f$ is $1$-Lipschitz it holds $\operatorname{md}((\varphi_i)_x) \geq \operatorname{md}((f \circ \varphi_i)_x)$ for all $i$ and almost all $x \in K_i$. Since $M$ is infinitesimally Euclidean, $\operatorname{md}((f \circ \varphi_i)_x)$ is an Euclidean norm for all $i$ and almost all $x \in K_i$. Since $\mu$ is Euclidean rigid, the second point above implies $\operatorname{md}((\varphi_i)_x) = \operatorname{md}((f \circ \varphi_i)_x)$ for all $i$ and almost all $x \in K_i$. Thus $f$ is an infinitesimal isometry and $S$ is infinitesimally Euclidean too.

We next apply Proposition~\ref{prop:almostisometry} by postcomposing with charts of $M$. Fix $y_0 \in M \setminus \partial M$ and let $D : M \to \R$ be the distance function to $y_0$, i.e.\ $D(y) = d(y_0,y)$. Fix $r > 0$ small enough such that $\B(y_0,2r) \cap \partial M = \emptyset$ and there exists a positively oriented bi-Lipschitz chart $\varphi : \oB(y_0,2r) \to \R^m$ onto an open subset of $\R^m$. By the slicing theory of \cite{AK} and assumption (b) we can further assume that
\[
T \res f^{-1}(\B(y_0,r)) = T \res f^{-1}(\oB(y_0,r)) \ \mbox{and}
\]
\[
 \quad \partial( T\res f^{-1}(\oB(y_0,r))) = \langle T, D\circ f, r\rangle \in \bI_{m-1}(X)
\]
holds for this $r$. Set $X' \defl f^{-1}(\B(y_0,r))$, $U \defl f^{-1}(\oB(y_0,r))$, $V \defl \varphi(\oB(y_0,r))$, $T' \defl T \res X'$ and $\pi \defl \varphi \circ f : X' \to \R^m$. Then
\begin{itemize}
	\item $T' \in \bI_m(X)$,
	\item $\pi_\# T' = \varphi_\#(\curr M \res \oB(y_0,r)) = \curr{V}$,
	\item $\spt(\partial T')$ is contained in $(D \circ f)^{-1}(r)$ which is disjoint from $U$.
\end{itemize}
We now apply Proposition~\ref{prop:almostisometry} with $X'$ and $T'$ in place of $X$ and $T$ respectively and $\pi$ as above. Assumptions (1), (2) and (3) of Proposition~\ref{prop:almostisometry} are clearly satisfied. It holds $f_\#\|T'\| = \cH^m \res \oB(y_0,r)$ by \eqref{eq:volpreserve} and because $\varphi$ is bi-Lipschitz there is some constant $C \geq 1$ such that
\[
C^{-1}\|\curr V\| \leq \pi_\#\|T'\| \leq C\|\curr V\| \ .
\]
The maximal function of $\|\curr V\|$ clearly satisfies $M_{\|\curr V\|} \leq 1$, hence (4) of Proposition~\ref{prop:almostisometry} holds too for some finite $L > 0$. Thus for all $y_0 \in M \setminus \partial M$ we find $0 < r < \dist(y_0,\partial M)$ such that $f : \spt(T) \cap f^{-1}(\oB(y_0,r)) \to \oB(y_0,r)$ is a homeomorphism and locally bi-Lipschitz. 

Collecting the local information, $f : X^\circ \defl \spt(T)\setminus\spt(\partial T) \to M^\circ \defl M \setminus \partial M$ is $1$-Lipschitz, surjective, open, locally bi-Lipschitz and an infinitesimal isometry. Since $f^{-1}(y) \cap X^\circ$ is a single point for $\cH^m$-almost all $y \in M^\circ$ and $f$ is open, $f : X^\circ \to M^\circ$ is injective and thus a homeomorphism which is locally bi-Lipschitz. Next we show that $f : X^\circ \to M^\circ$ also preserves the length of curves. Suppose that $\varphi : U \to X^\circ$ and $f\circ\varphi : U \to M^\circ$ are bi-Lipschitz charts defined on an open set $U \subset \R^m$. Because $f$ is an infinitesimal isometry it follows that $\operatorname{md}(\varphi_x) = \operatorname{md}((f\circ\varphi)_x)$ for almost all $x \in U$. We call this collection by $A \subset U$. If $\gamma : [0,1] \to U$ is a Lipschitz curve with $\cL^1(\gamma^{-1}(U\setminus A)) = 0$, then, for example by \cite[Theorem~4.1.6]{AT} it holds
\[
L(\varphi\circ\gamma) = \int_0^1 \operatorname{md}(\varphi_x)(\gamma'(t))\,dt = \int_0^1 \operatorname{md}((f\circ\varphi)_x)(\gamma'(t))\,dt = L(f\circ\varphi\circ\gamma) \ .
\]
Covering $M^\circ$ by countably many such bi-Lipschitz charts we find a set $N \subset M^\circ$ with $\cH^m(N) = 0$ such that $L(\gamma) = L(f^{-1} \circ \gamma)$ for all Lipschitz curves $\gamma : [0,1] \to M^\circ$ with $\cL^1(\gamma^{-1}(N)) = 0$. Because $M^\circ$ is an essential length space, this implies $d(x,y) \leq d(f(x),f(y))$ for all $x,y \in X^\circ$. The other inequality is clear because $f$ is $1$-Lipschitz. Thus $f : X^\circ \to M^\circ$ is an isometry. Because $M$ is the metric completion of $M^\circ$, and $X$ is complete, $f : \overline{X^\circ} \to M$ is an isometry too.
Now $X^\circ = \spt(T)\setminus\spt(\partial T)$ is dense in $\spt(T)$ because
\[
\|T\|^\mu(\spt(\partial T)) = \cH^m(f(\spt(\partial T))) \leq \cH^m(\partial M) = 0 \ ,
\]
by \eqref{eq:volpreserve} and assumption (b). Thus $\overline{X^\circ} = \spt(T)$ and hence $f : \spt(T) \to M$ is an isometry as claimed.
\end{proof}

\section{Counterexamples and comments}

In case $M$ has a boundary, it is not clear to the author if assumption (1) in Theorem~\ref{thm:rigidity} can be replaced by assuming that $(M,d)$ is an essential length space instead of $(M\setminus\partial M)$. In any case it can't be replaced by assuming $M$ (or $M\setminus \partial M$) to be a length space as the following example demonstrates.

\begin{ex}
	\label{ex:esslength}
Let $S^2$ be the standard Euclidean sphere in $\R^3$ with induced length metric $D$. Fix a great circle $C$ in $S^2$. A new metric $d$ on $S^2$ is defined by
\[
d(x,y) = \min\left(D(x,y), \inf_{v,w \in C} D(x,v) + \tfrac{1}{2}D(v,w) + D(w,y)\right) \ .
\]
The resulting metric space $(S^2,d)$ is denoted by $\cS$ and $f : S^2 \to \cS$ is the identity. The following statements are easy to check:
\begin{enumerate}
	\item $\cS$ is a geodesic space.
	\item $f$ is $1$-Lipschitz with $\frac{1}{2}$-Lipschitz inverse.
	\item $f$ is an infinitesimal isometry outside $C$ and thus area preserving.
	\item $f_\#\curr{S^2} = \curr{\cS}$ with $\bM(\curr S) = \bM(\curr{\cS})$ as a consequence of (2) and (3).
\end{enumerate}
$\cS$ is a geodesic Lipschitz manifold because of (1) and (2) but $f$ is not an isometry.
\end{ex}

A careful application of the Nash-Kuiper $C^1$-isometric embedding theorem applied to a sequence of Riemannian metrics on $S^2$ akin to \cite{LD} probably shows that $\cS$ can be realized isometrically as the length distance on some Lipschitz submanifold of $\R^3$. Note that \cite[Corollary~2.6]{LD} gives a path isometric embedding of $\cS$ into $\R^3$, but it is not clear that this embedding, in this particular situation (in general it is not), is also bi-Lipschitz.

We propose a different construction of a geodesic Lipschitz surface in $\R^3$ with the same properties listed for $\cS$ above. This serves as a counterexample to \cite[Question~8.1]{BCS}. The essential part of this surface is a Lipschitz graph over the $xy$-plane. Define $z : \R \to \R$ to be the periodic zigzag function
\[
z(t) \defl \min\{|t - n| : n \in \Z \} \ .
\]
For any integer $n \in \Z \setminus \{0\}$ the rescaled version $z_n : \R \to \R$ is defined by $z_n(t) \defl 2^{-n}z(t2^n)$. All these functions are piecewise linear with $|z_n'(t)| = 1$ for almost all $t$. Define the function $f : \R \times (0,\infty) \to \R$ by
\[
f(s,\lambda 2^n + (1-\lambda) 2^{n+1}) = \lambda z_n(s) + (1-\lambda)z_{n+1}(s)
\]
in case $n \in \Z$ and $\lambda \in [0,1]$. We extend $f$ to all of $\R^2$ by setting $f(x,0) = 0$ and $f(x,y) = f(x,-y)$ for $(x,y) \in \R \times (-\infty,0)$. It is easy to check that $f$ is Lipschitz and hence the graph
\[
M \defl \{(x,y,f(x,y)) : (x,y) \in \R^2\}
\]
equipped with the Euclidean distance $d$ of $\R^3$ is a Lipschitz surface. Let $d_i \geq d$ be the induced length distance on $M$. Because $f$ is Lipschitz, there is some $L > 0$ such that $d_i \leq Ld$. Let $I$ be the line segment in $M$ with endpoints $p = (0,0,0)$ and $q = (1,0,0)$. It holds $d(p,q) = d_i(p,q) = 1$. It can be shown that there is some $c > 0$ such that whenever $\gamma$ is a Lipschitz curve in $M$ connecting $p$ and $q$ with $\cH^1(\im(\gamma) \cap I) = 0$, then $L(\gamma) \geq c$. Thus $(M,d_i)$ is not an essential length space. If $d_e$ denotes the essential length distance on $M$ induced by $d$ (or equivalently $d_i$), then
\[
Ld \geq d_e \geq d_i \geq d \ .
\]
So the identity $g : (M,d_e) \to (M,d_i)$ is $1$-Lipschitz. $g$ is also volume preserving, since $M$ is piecewise smooth outside the $x$-axis $\{(x,0,0) : x \in \R\}$. But $g$ is not an isometry because $d_i(p,q) = 1 < d_e(p,q)$.

It is straight forward to modify $M$ in $\R^3$ so that the resulting space is bi-Lipschitz equivalent to $S^2$ with properties similar to $\cS$ above. For example we may restrict $M$ to $[-1,1]^2\times \R$ to obtain a compact Lipschitz surface with piecewise linear boundary which can be closed to obtain a Lipschitz sphere.

\end{document}